\documentclass[11pt]{article}

\usepackage{amsfonts,epsfig,fullpage}
\usepackage{kbordermatrix}
\usepackage{amsmath}
\usepackage[table]{xcolor}
\usepackage{multirow}
\usepackage{array}
\newcolumntype{?}{!{\vrule width 1pt}}
\usepackage{caption}
\usepackage{graphicx}

\newtheorem{theo}{Theorem}

\newtheorem{lem}{Lemma}
\newtheorem{claim}[lem]{Claim}

\newtheorem{coro}[lem]{Corollary}

\newtheorem{define}{Definition}
\newtheorem{fact}[lem]{Fact}

\newtheorem{alg}{Algorithm}
\newtheorem{proc}{Procedure}

\newcommand{\BE}{\begin{enumerate}} \newcommand{\EE}{\end{enumerate}}
\newcommand{\BI}{\begin{itemize}} \newcommand{\EI}{\end{itemize}}
\newcommand{\BDes}{\begin{description}}\newcommand{\EDes}{\end{description}
}
\newcommand{\BT}{\begin{theo}} \newcommand{\ET}{\end{theo}}
\newcommand{\BL}{\begin{lem}} \newcommand{\EL}{\end{lem}}
\newcommand{\BD}{\begin{define}} \newcommand{\ED}{\end{define}}
\newcommand{\BCM}{\begin{claim}} \newcommand{\ECM}{\end{claim}}
\newcommand{\BC}{\begin{coro}} \newcommand{\EC}{\end{coro}}
\newcommand{\BA}{\begin{alg}} \newcommand{\EA}{\end{alg}}
\newcommand{\BP}{\begin{proc}} \newcommand{\EP}{\end{proc}}

\def\FullBox{\hbox{\vrule width 8pt height 8pt depth 0pt}}
\newcommand{\qed}{\;\;\;\FullBox}
\newenvironment{proof}{\noindent{\bf Proof:~~}}{\(\qed\)}
\newcommand{\BPF}{\begin{proof}} \newcommand {\EPF}{\end{proof}}
\newenvironment{proofof}[1]{\noindent{\bf Proof of {#1}:~~}}{\(\qed\)}
\newcommand{\BPFOF}{\begin{proofof}} \newcommand {\EPFOF}{\end{proofof}}

\newcommand{\BEQ}{\begin{equation}} \newcommand{\EEQ}{\end{equation}}
\newcommand{\BEQN}{\begin{eqnarray}}\newcommand{\EEQN}{\end{eqnarray}}

\newlength{\saveparindent}
\newlength{\saveparskip}
\newcommand{\R}{{ R}}

\newcommand{\RB}{R_\mathbb{B}}

\renewcommand{\span}{\textsf{span}}

\renewcommand{\kbldelim}{(}
\renewcommand{\kbrdelim}{)}

\newtheorem{remark}{Remark}
\providecommand{\keywords}[1]{\textbf{\textit{Key Words---}} #1}
\providecommand{\subjclass}[1]{\textbf{\textit{Classification Codes---}} #1}

\newtheorem{propo}{Proposition}

\begin{document}

\title{Upper Bounds on the Boolean Rank of Kronecker Products}
\author{
Ishay Haviv\footnote{Research supported in part by the Israel Science Foundation (grant No.~1218/20).}\\
   The Academic College\\
    of Tel Aviv-Yaffo \\
    Tel Aviv, {\sc Israel}
   \and
    Michal Parnas \\
    The Academic College\\
    of Tel Aviv-Yaffo \\
    Tel Aviv, {\sc Israel} \\
    {\tt michalp@mta.ac.il}
}

\date{}

\maketitle

\abstract{
The {\em Boolean rank} of a $0,1$-matrix $A$, denoted $\RB(A)$, is the smallest number of monochromatic combinatorial rectangles needed to cover the $1$-entries of $A$.
In 1988, de Caen, Gregory, and Pullman asked if the Boolean rank of the Kronecker product $C_n \otimes C_n$ is strictly smaller than the square of $\RB(C_n)$,
where $C_n$ is the $n \times n$ matrix with zeros on the main diagonal and ones everywhere else (Carib.~Conf.~Comb.~\&~Comp.,~1988).
A positive answer was given by Watts for $n=4$ (Linear Alg. and its Appl.,~2001).
A result of Karchmer, Kushilevitz, and Nisan, motivated by direct-sum questions in non-deterministic communication complexity,
implies that the Boolean rank of $C_n \otimes C_n$ grows linearly in that of $C_n$ (SIAM J.~Disc.~Math.,~1995),
and thus $\RB(C_n \otimes C_n) < \RB(C_n)^2$ for every sufficiently large $n$.
Their proof relies on a probabilistic argument.

In this work, we present a general method for proving upper bounds on the Boolean rank of Kronecker products of $0,1$-matrices.
We use it to affirmatively settle the question of de Caen et al.~for all integers $n \geq 7$.
We further provide an explicit construction of a cover of $C_n \otimes C_n$, whose number of rectangles nearly matches the optimal asymptotic bound.
Our method for proving upper bounds on the Boolean rank of Kronecker products might find applications in different settings as well.
We express its potential applicability by extending it to the wider framework of {\em spanoids}, recently introduced by Dvir, Gopi, Gu, and Wigderson (SIAM J.~Comput.,~2020).
}\\

\iffalse

\fi

\section{Introduction}

The {\em Boolean rank} of a $0,1$-matrix $A$ of size $n \times m$, denoted $\RB(A)$, is the smallest number of monochromatic combinatorial rectangles needed to cover the $1$-entries of $A$.
Equivalently, $\RB(A)$ is the smallest integer $k$ for which $A = U \cdot V$ for two $0,1$-matrices $U$ and $V$ of sizes $n \times k$ and $k \times m$ respectively, where the operations are under Boolean arithmetic (namely, $0+x = x+0 = x$, $1+1=1$, $ 1\cdot 1 = 1$ and $x \cdot 0 = 0 \cdot x = 0$).
In terms of graph theory, $\RB(A)$ is the minimal number of bicliques needed to cover the edges of the bipartite graph whose `reduced' adjacency matrix is $A$.
Although the Boolean rank was studied extensively in various contexts over the years, its behavior is not fully understood, mainly because it is not defined over a field
(see, e.g., the survey~\cite{Monson} and the references therein). From a computational perspective, computing and even approximating the Boolean rank of a $0,1$-matrix are known to be NP-hard (see, e.g.,~\cite{orlin1977contentment,CHHK14}).

The Boolean rank of $0,1$-matrices plays a central role in the area of communication complexity.
In the two-party communication problem associated with a $0,1$-matrix $A$, a player Alice gets as input an index $i$ of a row of $A$
and a player Bob gets an index $j$ of a column. Their goal is to determine the value of $A_{i,j}$ using as few bits of communication as possible.
In the {\em non-deterministic} setting, the players act non-deterministically and their protocol should satisfy that $A_{i,j}=1$ if and only if there exists a transcript for which they output $1$.
It is well known that the Boolean rank of $A$ precisely measures the minimum number of transcripts needed to solve the problem.
Hence, the non-deterministic communication complexity of the problem associated with a matrix $A$ is $\lceil \log_2 \RB(A) \rceil$ (see, e.g.,~\cite{KN97}).

The current work is concerned with the Boolean rank of the Kronecker product $A \otimes B$ of $0,1$-matrices $A$ and $B$.
This product is defined as the block matrix whose $i,j$'th block is formed by multiplying the $i,j$'th entry $A_{i,j}$ of $A$ with the entire matrix $B$.
It is easy to verify that the Boolean rank is sub-multiplicative with respect to this product, that is,
\begin{eqnarray}\label{eq:sub_mul}
\RB(A \otimes B) \leq  \RB(A) \cdot \RB(B).
\end{eqnarray}
In 1988, it was asked by de Caen, Gregory, and Pullman~\cite{de1988boolean} whether the inequality can be strict (see also~\cite[Section~4.1]{Monson}).
As a candidate for strict inequality, they suggested to take $A$ and $B$ to be the $n \times n$ matrix $C_n$ that has zeros on the main diagonal and ones everywhere else
(note that $C_n$ is the `reduced' adjacency matrix of the {\em crown graph} on $2n$ vertices, defined as the complete bipartite graph $K_{n,n}$ from which the edges of a perfect matching have been removed).
They further proposed the problem of estimating the Boolean rank of $C_n \otimes C_n$ as a function of $n$.
The Boolean rank of $C_n$ itself was previously determined in~\cite{Caen2} using the celebrated Sperner's theorem. It was shown there that $R_{\mathbb{B}}(C_n) = \sigma(n)$,
where
\[\sigma(n) = \min \left\{k~ \Big | ~ n \leq \binom{k}{\lceil k/2\rceil}\right\}.\]
Note that $\sigma(n) = (1+o(1)) \cdot \log_2 n$.

A standard tool for proving lower bounds on the Boolean rank of $0,1$-matrices is based on the notion of {\em isolation sets} (also known as {\em fooling sets}).
An isolation set in a $0,1$-matrix $A$ is a subset of its $1$-entries, such that no two of them are in the same row or column of $A$, and no two of them belong to an all-one $2\times 2$ sub-matrix of $A$.
Since no two ones of an isolation set can belong to the same monochromatic rectangle, it follows that $R_{\mathbb{B}}(A) \geq i(A)$ for any $0,1$-matrix $A$,
where $i(A)$ denotes the largest size of an isolation set in $A$ (see, e.g.,~\cite{KN97}; see also~\cite{Watts06} for a somewhat better lower bound based on a fractional variant of the Boolean rank).
It was observed in~\cite{de1988boolean} that $R_{\mathbb{B}}(A \otimes B) \geq \max\{i(A) \cdot R_{\mathbb{B}}(B),R_{\mathbb{B}}(A) \cdot i(B)\}$.
Since $i(C_n)=3$ for every $n \geq 3$  (see, e.g.,~\cite{de1988boolean}), this implies that  $\RB(C_n \otimes C_n) \geq 3 \cdot \RB(C_n)$ for every such $n$.

In 2001, Watts~\cite{watts2001boolean} proved that $\RB(C_4 \otimes C_4) = 12$, which by $\RB(C_4) = \sigma(4) = 4$ implies that the inequality in~\eqref{eq:sub_mul} is strict for $A=B=C_4$.
Watts further observed that this example is minimal in terms of dimension.
As for the asymptotic behavior of the Boolean rank of $C_n \otimes C_n$, a probabilistic argument of Karchmer, Kushilevitz, and Nisan~\cite{karchmer1995fractional} implies that there exists a constant $c>0$, such that for every sufficiently large integer $n$ it holds that $\RB(C_n \otimes C_n) \leq c \cdot \log_2 n$. This, in particular, yields that the inequality in~\eqref{eq:sub_mul} is strict for $A=B=C_n$ whenever $n$ is sufficiently large.
While the result of~\cite{karchmer1995fractional} is more general, for the matrix $C_n \otimes C_n$ their bound is essentially obtained by uniformly picking random rectangles of ones in $C_n \otimes C_n$ with maximum size.
Since these rectangles cover the $1$-entries of $C_n \otimes C_n$ evenly, and since each of them covers a constant fraction of its $O(n^4)$ $1$-entries, $O(\log_2 n)$ such rectangles suffice to cover the matrix $C_n \otimes C_n$.

The study of the Boolean rank of Kronecker products of $0,1$-matrices is motivated by direct-sum questions in communication complexity (see, e.g.,~\cite{FederKNN95,karchmer1995fractional} and~\cite[Chapter~4]{KN97}).
Indeed, the matrix $A \otimes A$ is associated with the communication complexity problem in which Alice gets as input two indices $i_1,i_2$ of rows of $A$, Bob gets two indices $j_1,j_2$ of columns of $A$, and their goal is to determine $A_{i_1,j_1} \cdot A_{i_2,j_2}$.
Equivalently, Alice and Bob get two instances of the problem associated with $A$ and their goal is to decide if {\em both} of them are yes-instances (that is, to compute the AND operation applied to the outcomes of the two instances).

Consider for example the matrix $C_n$, and observe that it represents the two-party communication problem of non-equality over a domain of size $n$. The aforementioned result of~\cite{karchmer1995fractional} interestingly implies that the non-deterministic communication complexity required for determining the AND of two instances of this problem is larger only by an absolute constant than that of a single instance.
In fact, the results of~\cite{karchmer1995fractional} extend to the more general scenario where the number of instances that the players get is arbitrary and their goal is to decide if all of them are yes-instances.
It turns out that when the number of instances grows to infinity, the {\em amortized} non-deterministic communication complexity is determined by a {\em fractional} variant of the Boolean rank.
For the matrix $C_n$, the latter is bounded from above by a constant independent of $n$ (see~\cite[Section~2.3 and Example~2]{karchmer1995fractional} and~\cite[Chapters~2.4 and~4.1.2]{KN97}).

\subsection{Our Contribution}

In this work, we present a general method for proving upper bounds on the Boolean rank of Kronecker products of $0,1$-matrices, as described in the following theorem.
Here, a collection of $0,1$-matrices is said to {\em cover} a matrix if their sum under Boolean arithmetic is equal to that matrix.

\BT\label{theorem:RankTensorIntro}
Let $A$ and $B$ be two $0,1$-matrices.
Suppose that there exist two families of $0,1$-matrices, $\mathcal{M} = \{M_{1},...,M_{s}\}$ and $\mathcal{N} = \{N_{1},...,N_{s}\}$, such that
\begin{enumerate}
  \item $\mathcal{M}$ is a cover of $A$, and
  \item for every $i,j$ such that $A_{i,j}=1$, the matrices $N_t$ with ${(M_t)}_{i,j} = 1$ form a cover of $B$.
\end{enumerate}
Then,
\[\RB(A \otimes B) \leq \sum_{t=1}^{s} \RB(M_{t}) \cdot \RB(N_{t}).\]
\ET
We remark that here and throughout the paper, we allow repetitions in the families $\mathcal{M}$ and $\mathcal{N}$, which thus can be viewed as sequences of not necessarily distinct matrices.

Theorem~\ref{theorem:RankTensorIntro} essentially reduces the challenge of proving upper bounds on the Boolean rank of $A \otimes B$ to finding two families $\mathcal{M}$ and $\mathcal{N}$ of low-rank matrices satisfying the conditions of the theorem.
It turns out that a useful way to obtain these families is to construct, for each of the matrices $A$ and $B$, a collection of low-rank matrices
such that every $1$-entry is covered by many of them (see Corollary~\ref{cor:RankAxB_k}).
We note that Theorem~\ref{theorem:RankTensorIntro} is inspired by the approach taken by Watts~\cite{watts2001boolean} in her upper bound on the Boolean rank of $C_4 \otimes C_4$, where she essentially considers a special case of the theorem with families of matrices of Boolean rank  $2$.

It is worth mentioning that the bound given in Theorem~\ref{theorem:RankTensorIntro} on the Boolean rank of $A \otimes B$ is tight.
Namely, if $\RB(A \otimes B)=s$ then there exist families $\mathcal{M}$ and $\mathcal{N}$ as in the theorem, each with $s$ matrices of Boolean rank $1$, attaining an upper bound of $s$ on the Boolean rank of $A \otimes B$ (see Lemma~\ref{lemma:lower}). We show that the existence of these families implies a general {\em lower} bound on the Boolean rank of Kronecker products of $0,1$-matrices.
See Theorem~\ref{thm:lower} for the precise statement. As an application, we show that the Boolean rank of $C_n \otimes C_n$ is at least $4-o(1)$ times the Boolean rank of $C_n$, slightly improving on the multiplicative constant $3$ derived above using the isolation sets of $C_n$ (see Corollary~\ref{cor:lower}).

We apply Theorem~\ref{theorem:RankTensorIntro} to study the Boolean rank of the matrix $C_n \otimes C_n$ and to provide economical covers of $C_n \otimes C_n$ by monochromatic combinatorial rectangles.
Our approach allows us to obtain non-trivial upper bounds on the Boolean rank of $C_n \otimes C_n$ already for small values of $n$.
This is used to show that the Boolean rank of $C_n \otimes C_n$ is strictly smaller than the square of the Boolean rank of $C_n$ for all integers $n \geq 7$. More generally, we prove the following.
\BT\label{thm:n>=7Intro}
For all integers $n,m \geq 7$,~ $\RB(C_n \otimes C_m) < \RB(C_n) \cdot \RB(C_m)$.
\ET
Combined with the results of Watts~\cite{watts2001boolean}, Theorem~\ref{thm:n>=7Intro} settles the question of whether a strict inequality holds in~\eqref{eq:sub_mul} for $A=B=C_n$, for all integers $n$ besides $5$ and $6$ (see Section~\ref{subsub:examples}).

We further demonstrate that our method can be used to obtain asymptotically economical covers.
We provide a cover of the matrix $C_n \otimes C_n$ with $O(k \log_2 k)$ rectangles for $k = \RB(C_n)$, matching the bound of~\cite{karchmer1995fractional} up to a $\log_2 \log_2 n$ multiplicative term (see Theorem~\ref{thm:mlogm}).
The construction involves combinatorial and algebraic arguments, and in contrast to the probabilistic construction of~\cite{karchmer1995fractional}, is {\em explicit}, that is, can be computed by an efficient deterministic algorithm.
From the communication complexity point of view, the explicitness of the cover means that the players can separately and efficiently compute the cover needed for their non-deterministic protocol, and in particular, they do not have to share any randomness or other information in advance.
We remark that we were recently informed in a personal communication with Kushilevitz~\cite{KushilevitzPersonal} that it is also possible to explicitly construct a cover of $C_n \otimes C_n$ with $O(k)$ rectangles for $k = \RB(C_n)$ (where again the result provides a gap between $\RB(C_n \otimes C_n)$ and $\RB(C_n)^2$ for any sufficiently large $n$).

Our method for proving upper bounds on the Boolean rank of Kronecker products, as given in Theorem~\ref{theorem:RankTensorIntro}, might find additional applications.
Indeed, it can be applied to additional families of matrices other than $C_n$ (see Section~\ref{subsub:examples} for an example).
Moreover, it is not limited specifically to the Boolean rank and can be useful for other rank functions in other contexts as well.
To demonstrate its potential applicability, we consider the notion of {\em spanoids}, recently introduced by Dvir, Gopi, Gu, and Wigderson~\cite{DvirGGW20}.
A spanoid is a simple logical inference structure that captures various mathematical objects with applications in areas of research like combinatorics, algebra, statistical physics, network theory, and coding theory.
A central parameter of a spanoid is its rank and a basic operation on spanoids is a product.
It turns out that the Boolean rank of $0,1$-matrices can be represented as the rank of a special family of spanoids.
We extend Theorem~\ref{theorem:RankTensorIntro} to the wider framework of spanoids, providing a general technique for proving upper bounds on the rank of products of spanoids.

\subsection{Overview of Proofs}
\label{sec-outline}
Our explicit bounds on the Boolean rank of the matrices $C_n \otimes C_n$ rely on the general method presented in Theorem~\ref{theorem:RankTensorIntro}.
To get some intuition for the method, suppose that for some $0,1$-matrix $A$ there exist three matrices $A_1,A_2,A_3$ of Boolean rank $a_1,a_2,a_3$ respectively, such that every two of them form a cover of $A$. In such a case, the matrix $A$ is said to be {\em $(a_1,a_2,a_3)$-coverable}. Similarly, suppose that a $0,1$-matrix $B$ is $(b_1,b_2,b_3)$-coverable, and let $B_1,B_2,B_3$ be the corresponding matrices. It can be shown that these matrices satisfy, under Boolean arithmetic, that
\begin{eqnarray*}%\label{eq:overviewAxB}
A \otimes B = A_1 \otimes B_1 + A_2 \otimes B_2 + A_3 \otimes B_3,
\end{eqnarray*}
as implied by comparing every block of the matrix $A \otimes B$ to its corresponding block in the right hand side (see the proof of Theorem~\ref{theorem:RankTensorIntro}).
It thus follows, using the sub-additivity and sub-multiplicativity of the Boolean rank, that
\begin{eqnarray}\label{eq:overviewRank}
\RB(A \otimes B) \leq a_1 \cdot b_1 + a_2 \cdot b_2 + a_3 \cdot b_3.
\end{eqnarray}
This reduces proving upper bounds on the Boolean rank of the product $A \otimes B$ to finding triples of low-rank matrices as above for $A$ and $B$.

To demonstrate the above approach, let us consider the case where both $A$ and $B$ are equal to the matrix $C_n$.
Notice that to obtain a non-trivial upper bound on the Boolean rank of $C_n \otimes C_n$, namely, a bound strictly smaller than $\RB(C_n)^2$, it suffices to show that $C_n$ is $(1,k-1,k-1)$-coverable for $k = \RB(C_n)$. Indeed, by symmetry this implies that $C_n$ is also $(k-1,1,k-1)$-coverable, so applying~\eqref{eq:overviewRank} we obtain that
\[\RB(C_n \otimes C_n) \leq 1 \cdot (k-1) + (k-1) \cdot 1 + (k-1) \cdot (k-1) = k^2-1 < \RB(C_n)^2.\]

We next describe how to explicitly find for $C_n$ the three required matrices $A_1,A_2,A_3$.
Suppose for simplicity that $n$ has the form $n = \binom{k}{ k/2}$ for an even integer $k$, and recall that $k = \RB(C_n)$.
It will be convenient to identify the rows of $C_n$ with all the $k/2$-subsets of $[k]$ and to identify its columns with their complements.
Namely, if the $i$'th row of $C_n$ corresponds to a $k/2$-subset $F_i$ of $[k]$ then the $i$'th column of $C_n$ corresponds to $\overline{F_i} = [k] \setminus F_i$.
This provides a natural cover of $A$ by $k$ matrices of Boolean rank $1$, where the $j$'th matrix has ones in the entries whose rows and columns correspond to sets that include $j$ ($j \in [k]$).
Now, defining $A_1 $ to be the first such matrix and $A_2$ to be the sum of the remaining $k-1$ matrices,
we get two matrices of Boolean rank $1$ and $k-1$ that together cover $C_n$.
It remains to define another matrix $A_3$ of Boolean rank $k-1$, such that both the pairs $A_1,A_3$ and $A_2,A_3$ cover $C_n$ as well.

In order to define the third matrix $A_3$, we need the simple observation that the order of the sets assigned above to the rows of $C_n$ can be arbitrary.
In fact, applying any permutation $h$ to the $k/2$-subsets of $[k]$ results in another cover of $C_n$ by $k$ matrices,
for which one can consider the matrix $A_3$ of Boolean rank $k-1$, which again includes all matrices but the first one.
However, the permutation $h$ should be chosen in a way that ensures that the pairs $A_1,A_3$ and $A_2,A_3$ cover $C_n$.

On the one hand, for $A_1,A_3$ to cover $C_n$, it suffices to require $h$ to {\em preserve} the element $1$ in the sets, that is, to satisfy $1 \in F \Leftrightarrow 1 \in h(F)$ for every set $F$,
because this implies that the two covers share the same first matrix.
On the other hand, for $A_2,A_3$ to cover $C_n$, we need $h$ to satisfy that whenever an entry of $C_n$ is not covered by $A_2$, that is, its row and column sets $F_i,\overline{F_{i'}}$ satisfy
$F_i \cap \overline{F_{i'}} = \{1\}$, the intersection $h(F_i) \cap \overline{h(F_{i'})}$ includes some element different from $1$. Intuitively, $h$ is required to {\em shift} such an intersection from $\{1\}$ to some other non-empty set. It can be shown that a permutation satisfying these two conditions indeed exists, allowing us to obtain the required triple $A_1,A_2,A_3$ (see Section~\ref{sec:strictly} for a full and extended argument).

The above idea suffices to prove that the Boolean rank of $C_n \otimes C_n$ is strictly smaller than the square of the Boolean rank of $C_n$ for all integers $n \geq 7$.
However, this argument is inherently limited in the upper bound that it can yield for the Boolean rank of $C_n \otimes C_n$. The reason is that if $C_n$ is $(a_1,a_2,a_3)$-coverable then for every distinct $i,j \in [3]$ we must have $a_i + a_j \geq k$, which implies that the bound that follows from~\eqref{eq:overviewRank} cannot be lower than $3k^2/4$.
To get below this bound, we need a generalized version of the approach presented here.
Namely, to prove an upper bound on the Boolean rank of $A \otimes B$, we need a collection of $s$ matrices for each of $A$ and $B$,
such that every $\lceil s/2 \rceil$ of them cover $A$ and $B$ respectively.
Letting the number of matrices $s$ be much larger than $3$ allows us to obtain more economical constructions of covers of $C_n \otimes C_n$ (see Corollary~\ref{cor:RankAxB_k}).

To construct an asymptotically economical cover of $C_n \otimes C_n$, we apply the above approach and provide $s$ matrices such that every $\lceil s/2 \rceil$ of them cover $C_n$, where $s = \Theta(k/\log k)$. The basic strategy is similar to the one presented above. We again view the matrix $C_n$ as the matrix that represents the intersections of all $k/2$-subsets of $[k]$.
However, while in the above construction we used two different ways to assign sets to the rows of $C_n$ (an arbitrary one and another one obtained by the permutation $h$), here we need a much larger number of permutations, each of which induces a cover of size $k$ of $C_n$ used to define a matrix of Boolean rank roughly $2k/s$. Note that this is the smallest possible rank of the matrices in such a collection, given that every $\lceil s/2 \rceil$ of them cover the matrix $C_n$ (whose Boolean rank is $k$).

The construction of these permutations is based on a family of functions with the property that only few of them can share a common collision (i.e., two inputs that are mapped to the same output). This family of functions is explicitly constructed using an algebraic argument over finite fields (see Lemma~\ref{lem:g's_prime_new}). This gives us the family of matrices on which we apply Corollary~\ref{cor:RankAxB_k} to derive the guaranteed explicit cover (see Section~\ref{sec:asymptotic}).

\subsection{Outline}
The rest of the paper is organized as follows.
In Section~\ref{sec:preliminaries}, we gather several definitions and claims needed throughout the paper.
In Section~\ref{sec:Bounds}, we present our method for proving upper bounds on the Boolean rank of Kronecker products of $0,1$-matrices and prove Theorem~\ref{theorem:RankTensorIntro}.
We also show there that the upper bound provided by the theorem is tight and use it to derive a lower bound on the Boolean rank of $C_n \otimes C_n$.
In Section~\ref{sec:Covers_Cn}, we use Theorem~\ref{theorem:RankTensorIntro} to obtain our explicit constructions of economical covers of the matrices $C_n \otimes C_n$.
We first show that for all integers $n \geq 7$, there exists a cover with strictly less than $\RB(C_n)^2$ rectangles, and more generally, prove Theorem~\ref{thm:n>=7Intro}.
Then, we provide an explicit construction of a cover of $C_n \otimes C_n$ for the asymptotic case.
Finally, in Section~\ref{sec:spanoids}, we provide a brief introduction to the topic of spanoids introduced in~\cite{DvirGGW20} and generalize Theorem~\ref{theorem:RankTensorIntro} to this setting.

\section{Preliminaries}\label{sec:preliminaries}

We start with a few definitions and notation that will be needed throughout the paper.
Entry $i,j$ of a matrix $M$ will be denoted by $M_{i,j}$, and if a matrix $M_t$ has a subscript then its $i,j$'th entry will be denoted by $(M_t)_{i,j}$.
For two $0,1$-matrices $M$ and $A$, we write $M \preceq A$ if $M_{i,j} \leq A_{i,j}$ for all $i,j$.

Let $A$ be a $0,1$-matrix of size $n \times m$.
A {\em combinatorial rectangle} of $A$ is a subset of $[n] \times [m]$ of the form $X \times Y$, where $X \subseteq [n]$ and $Y \subseteq [m]$.
The rectangle is {\em monochromatic} if all the values $A_{i,j}$ with $i \in X$ and $j \in Y$ are equal.
As defined earlier, the {\em Boolean rank} of $A$, denoted $\RB(A)$, is the smallest number of monochromatic combinatorial rectangles needed to cover the $1$-entries of $A$. Note that every monochromatic combinatorial rectangle of ones can be viewed as a matrix of Boolean rank $1$.

A collection $\mathcal{M} = \{M_1,\ldots,M_s\}$ of $0,1$-matrices is said to form a {\em cover} of $A$ (or to {\em cover} $A$) if for every $i,j$ it holds that $A_{i,j}=1$ if and only if there exists some $t \in [s]$ for which $(M_t)_{i,j}=1$.
Equivalently, the matrices of $\mathcal{M}$ should satisfy, under Boolean arithmetic, that $A = \sum_{t \in [s]}{M_t}$. Note that $M_t \preceq A$ for all $t \in [s]$.

The Boolean rank of $A$ can also be defined as the smallest integer $k$ for which one can assign subsets of $[k]$ to its rows and columns,
where if $X_i$ and $Y_j$ are the subsets assigned to row $i$ and column $j$ respectively, then $X_i \cap Y_j \neq \emptyset$ if and only if $A_{i,j} = 1$.
Notice that for every $t \in [k]$, the set of all entries $A_{i,j}$ for which $t \in X_i \cap Y_j$ defines a monochromatic combinatorial rectangle of ones.
Since these $k$ rectangles cover all the $1$-entries of $A$, we get a cover of $A$ of size $k$.

For an illustration of these concepts, consider the following example.
\begin{equation*}
\kbordermatrix{   &   \{2\} &   \{2\} &   \{1\} &  \{1\} \\
                             \{1\}     & 0 & 0 & 1 & 1   \\
                             \{1\}     & 0 & 0 & 1 & 1   \\
                             \{2\}     & 1 & 1 & 0 & 0   \\
                             \{2\}     & 1 & 1 & 0 & 0  }
                             ~~~+~~~
                             \kbordermatrix{   &   \{2\} &   \{1\} &   \{1\} &  \{2\} \\
                             \{1\}     & 0 & 1 & 1 & 0   \\
                             \{2\}     & 1 & 0 & 0 & 1   \\
                             \{2\}     & 1 & 0 & 0 & 1   \\
                             \{1\}     & 0 & 1 & 1 & 0  }
                             ~~~=
                             \kbordermatrix{   &   ~ &  ~ &   ~&  ~ \\
                             ~     & 0 & 1 & 1 & 1   \\
                             ~     & 1 & 0 & 1 & 1   \\
                             ~     & 1 & 1 & 0 & 1   \\
                             ~     & 1 & 1 & 1 & 0  }.
\end{equation*}
Here, the two matrices on the left form a cover of the matrix on the right. Above and to the left of these two matrices are written the subsets defining them.
Each of the two matrices on the left has Boolean rank $2$, and the matrix on the right has Boolean rank $4$.

For an integer $n$, let $C_n$ be the $n \times n$ matrix that has zeros on the main diagonal and ones everywhere else.
The following lemma, given in~\cite{Caen2}, determines its Boolean rank. Recall that $\sigma(n)$ is the smallest integer $k$ satisfying $n \leq \binom{k}{\lceil k/2 \rceil}$.
%We include a quick proof for completeness.
\BL\label{lem:R(C_n)}(\cite{Caen2})
The Boolean rank of $C_n$ is $\sigma(n)$.
\EL

The {\em Kronecker product} $A \otimes B$ of $0,1$-matrices $A$ and $B$ is the block matrix whose $i,j$'th block is formed by multiplying the $i,j$'th entry $A_{i,j}$ of $A$ with the entire matrix $B$.
We need the following simple claim.
\begin{claim}\label{claim:rectangles}
Let $A$ and $B$ be $0,1$-matrices, and let $M$ be a $0,1$-matrix of Boolean rank $1$ such that $M \preceq A \otimes B$.
Then, there exist $0,1$-matrices $M_A$ and $M_B$ of Boolean rank $1$, satisfying $M_A \preceq A$, $M_B \preceq B$, and $M \preceq M_A \otimes M_B$.
\end{claim}

\BPF
By definition, every entry of $A \otimes B$ can be indexed by a $4$-tuple $(i_A,j_A,i_B,j_B)$, so that
\begin{eqnarray}\label{eq:AxB_claim}
(A \otimes B)_{i_A,j_A,i_B,j_B} = A_{i_A,j_A} \cdot B_{i_B,j_B}.
\end{eqnarray}
Suppose that $M \preceq A \otimes B$ and that $M$ has Boolean rank $1$.
Define the $0,1$-matrix $M_A$ such that $(M_A)_{i_A,j_A} = 1$ if and only if there exist $i_B,j_B$ with $M_{i_A,j_A,i_B,j_B} = 1$.
Similarly, define the $0,1$-matrix $M_B$ such that $(M_B)_{i_B,j_B} = 1$ if and only if there exist $i_A,j_A$ with $M_{i_A,j_A,i_B,j_B} = 1$.
Since $M \preceq A \otimes B$ and using~\eqref{eq:AxB_claim}, it follows that $M_A \preceq A$ and $M_B \preceq B$.
To prove that $M_A$ has Boolean rank $1$, it suffices to show that $M_A$ is nonzero and that if $(M_A)_{i_A,j_A} = 1$ and $(M_A)_{i'_A,j'_A} = 1$ then $(M_A)_{i_A,j'_A} = 1$ and $(M_A)_{i'_A,j_A} = 1$.

Since $M$ is nonzero, it is clear that $M_A$ is nonzero as well.
Suppose now that $(M_A)_{i_A,j_A} = 1$ and $(M_A)_{i'_A,j'_A} = 1$.
By the definition of $M_A$, this implies that there exist $i_B,j_B,i'_B,j'_B$ such that $M_{i_A,j_A,i_B,j_B} = 1$ and $M_{i'_A,j'_A,i'_B,j'_B} = 1$.
Since $M$ has Boolean rank $1$, it follows that $M_{i_A,j'_A,i_B,j'_B} = 1$ and $M_{i'_A,j_A,i'_B,j_B} = 1$.
By the definition of $M_A$, we obtain that $(M_A)_{i_A,j'_A} = 1$ and $(M_A)_{i'_A,j_A} = 1$, as required. By symmetry, it also follows that $M_B$ has Boolean rank $1$.

We finally claim that $M \preceq M_A \otimes M_B$.
Indeed, suppose that $M_{i_A,j_A,i_B,j_B} = 1$. By the definition of $M_A$ and $M_B$, it follows that $(M_A)_{i_A,j_A} = 1$ and $(M_B)_{i_B,j_B} = 1$, and therefore $(M_A \otimes M_B)_{i_A,j_A,i_A,j_B} = 1$, so we are done.
\EPF \\

We end this section with the following well-known fact.
\begin{fact}\label{fact:central_binom}
For every even integer $r$, $\binom{r}{r/2} \geq \frac{2^r}{\sqrt{2r}}$.
\end{fact}

\section{The Boolean Rank of Kronecker Products}\label{sec:Bounds}

\subsection{Upper Bound}

We restate and prove the following theorem, which presents our general method for proving upper bounds on the Boolean rank of Kronecker products of matrices.\\

\noindent
{\bf Theorem 1}
{\em
%\BT\label{theorem:RankTensor}
Let $A$ and $B$ be two $0,1$-matrices.
Suppose that there exist two families of $0,1$-matrices, $\mathcal{M} = \{M_{1},...,M_{s}\}$ and $\mathcal{N} = \{N_{1},...,N_{s}\}$, such that
\begin{enumerate}
  \item $\mathcal{M}$ is a cover of $A$, and
  \item for every $i,j$ such that $A_{i,j}=1$, the matrices $N_t$ for which ${(M_t)}_{i,j} = 1$ form a cover of $B$.
\end{enumerate}
Then,
\[\RB(A \otimes B) \leq \sum_{t=1}^{s} \RB(M_{t}) \cdot \RB(N_{t}).\]
}%\ET

\BPF
We first show that
\begin{equation}\label{equal2}
A \otimes B =  \sum_{t = 1}^s (M_{t} \otimes N_{t}),
\end{equation}
where the sum is under Boolean arithmetic.
Recall that the matrix $A \otimes B$ can be viewed as a matrix of blocks, where the block of $A \otimes B$ indexed by $i,j$ is $A_{i,j} \cdot B$.
To prove~\eqref{equal2}, it suffices to compare the matrices block by block. So consider an arbitrary block indexed by $i,j$.

Suppose first that $A_{i,j}=0$. In this case, the $i,j$'th block of $A \otimes B$ is the zero matrix.
Since $\mathcal{M}$ is a cover of $A$, for every $t \in [s]$ we have $(M_t)_{i,j}=0$.
Hence, the $i,j$'th block of $M_t \otimes N_t$ is also the zero matrix for all $t \in [s]$. It follows that the corresponding block of their sum is zero as well.

Suppose now that $A_{i,j}=1$. Here, the $i,j$'th block of $A \otimes B$ is equal to $B$.
On the other hand, the $i,j$'th block of the matrix $\sum_{t = 1}^s (M_{t} \otimes N_{t})$
is precisely the sum of the matrices $N_t$ for which $(M_t)_{i,j}=1$. By assumption, these matrices form a cover of $B$. Hence, this block is equal to $B$, as required.

Finally, combining~\eqref{equal2} with the sub-additivity and sub-multiplicativity of the Boolean rank, we obtain that
\[\RB(A \otimes B) = \RB \bigg ( \sum_{t = 1}^s (M_{t} \otimes N_{t}) \bigg ) \leq \sum_{t=1}^{s}{\RB(M_{t} \otimes N_{t})}  \leq \sum_{t=1}^{s} \RB(M_{t}) \cdot \RB(N_{t}),\]
completing the proof.
\EPF\\

\begin{remark}\label{remark:explicitThm1}
We note that the proof of Theorem~\ref{theorem:RankTensorIntro}, in addition to providing an upper bound on the Boolean rank of $A \otimes B$, shows that the matrices of the given families $\mathcal{M} = \{M_1,\ldots,M_s\}$ and $\mathcal{N}=\{N_1,\ldots,N_s\}$ can in certain cases be used to explicitly construct a cover of $A \otimes B$ by monochromatic combinatorial rectangles of ones. To see this, suppose that every matrix $M_t$ is given as a sum of matrices $M_t(i)$ of Boolean rank $1$ for $i=1,\ldots,\RB(M_i)$, and that every matrix $N_t$ is given as a sum of matrices $N_t(j)$ of Boolean rank $1$ for $j=1,\ldots,\RB(N_i)$.
It can be seen that for every $t \in [s]$ the matrices $M_t(i) \otimes N_t(j)$ form a cover of $M_t \otimes N_t$.
It thus follows from~\eqref{equal2} that they all form a cover of $A \otimes B$ with $\sum_{t=1}^{s}{\RB(M_t) \cdot \RB(N_t)}$ matrices of Boolean rank $1$.
\end{remark}

In order to prove an upper bound on the Boolean rank of $A \otimes B$ using Theorem~\ref{theorem:RankTensorIntro},
one has to find two families $\mathcal{M}$ and $\mathcal{N}$ of matrices which satisfy the conditions of the theorem.
The following corollary suggests a simple property of such families that suffices for this purpose.

\begin{coro}\label{cor:RankAxB_k}
Let $A$ and $B$ be two $0,1$-matrices.
Suppose that there exist two families of $0,1$-matrices, $\mathcal{M} = \{M_{1},...,M_{s}\}$ and $\mathcal{N} = \{N_{1},...,N_{s}\}$,
such that every $\lceil s/2 \rceil$ matrices of $\mathcal{M}$ cover $A$ and every $\lceil s/2 \rceil$ matrices of $\mathcal{N}$ cover $B$.
Then,
\[\RB(A \otimes B) \leq \sum_{t=1}^{s} \RB(M_{t}) \cdot \RB(N_{t}).\]
In particular,
\[\RB(A \otimes A) \leq \sum_{t=1}^{s} \RB(M_{t})^2.\]
\end{coro}

\BPF
Suppose that every $\lceil s/2 \rceil$ matrices of $\mathcal{M}$ cover $A$ and that every $\lceil s/2 \rceil$ matrices of $\mathcal{N}$ cover $B$.
We argue that these families of matrices satisfy the conditions of Theorem~\ref{theorem:RankTensorIntro}.

Clearly, $\mathcal{M}$ is a cover of $A$.
We next observe that for every $i,j$ such that $A_{i,j}=1$, the number of matrices $M_t$ satisfying $(M_t)_{i,j}=1$ is at least $\lceil s / 2 \rceil$.
Indeed, if their number was smaller than $\lceil s / 2 \rceil$ then all the other matrices, whose number is larger than $s- \lceil s / 2 \rceil = \lfloor s / 2 \rfloor$, would have zeros in their $i,j$'th entry, in contradiction to the assumption that they form a cover of $A$.
It thus follows that for every $i,j$ such that $A_{i,j}=1$, the number of matrices $N_t$ for which $(M_t)_{i,j}=1$ is at least $\lceil s / 2 \rceil$, so by assumption they cover the matrix $B$.
This allows us to apply Theorem~\ref{theorem:RankTensorIntro} and to complete the proof.
\EPF\\

Let us consider now the case where the number $s$ of matrices in the families $\mathcal{M}$ and $\mathcal{N}$ is $3$, as given in the following definition.
\BD
A $0,1$-matrix $A$ is said to be {\em $(a_1,a_2,a_3)$-coverable} if there exist three matrices $A_1$, $A_2$, and $A_3$ satisfying $\RB(A_i) \leq a_i$ for all $i \in [3]$, such that every two of them cover $A$.
\ED
By Corollary~\ref{cor:RankAxB_k}, we immediately get the following statement, that will be useful to us later.
\begin{coro}\label{cor:RankAxB_3}
Let $A$ and $B$ be two $0,1$-matrices.
Suppose that $A$ is $(a_1,a_2,a_3)$-coverable and that $B$ is $(b_1,b_2,b_3)$-coverable.
Then, $\RB(A \otimes B) \leq a_1 \cdot b_1 + a_2 \cdot b_2 + a_3 \cdot b_3$.
\end{coro}

We finally remark that when Theorem~\ref{theorem:RankTensorIntro} is applied, the order of the matrices of the families $\mathcal{M}$ and $\mathcal{N}$ may affect the obtained bound.
This turns out to be crucial when the Boolean ranks of the matrices in the families are quite different from one another.
For example, if a matrix $A$ is $(a_1,a_2,a_3)$-coverable then Corollary~\ref{cor:RankAxB_3} implies that
\[\RB(A \otimes A) \leq a_1^2 + a_2^2 + a_3^2.\]
However, the matrix $A$ is also, say, $(a_2,a_1,a_3)$-coverable. Hence, the same corollary also gives
\[\RB(A \otimes A) \leq 2 a_1 a_2 + a_3^2,\]
which might sometimes provide a better upper bound on  $\RB(A \otimes A)$.

\subsection{Lower Bound}\label{sec:lower}

The following lemma shows that Theorem~\ref{theorem:RankTensorIntro} is tight, in the sense that for every two $0,1$-matrices there exist families of matrices
as in the theorem, attaining a tight upper bound on the Boolean rank of their Kronecker product.

\BL
\label{lemma:lower}
Let $A$ and $B$ be two $0,1$-matrices, and let $s = \RB(A \otimes B)$.
Then, there exist two families of $0,1$-matrices of Boolean rank $1$, $\mathcal{M} = \{M_{1},...,M_{s}\}$ and $\mathcal{N} = \{N_{1},...,N_{s}\}$, such that
\begin{enumerate}
  \item $\mathcal{M}$ is a cover of $A$, and
  \item for every $i,j$ such that $A_{i,j}=1$, the matrices $N_t$ for which ${(M_t)}_{i,j} = 1$ form a cover of $B$.
\end{enumerate}
\EL

\BPF
Since $s = \RB(A \otimes B)$, it follows that the matrix $A \otimes B$ has a cover of $s$ matrices $P_1,\ldots,P_s$ of Boolean rank $1$.
By Claim~\ref{claim:rectangles}, for every $t \in [s]$ there exist two $0,1$-matrices $M_t$ and $N_t$ of Boolean rank $1$, satisfying $M_t \preceq A$, $N_t \preceq B$, and $P_t \preceq M_t \otimes N_t$.
It thus follows that the matrices $M_t \otimes N_t$ with $t \in [s]$ also form a cover of $A \otimes B$, that is, $A \otimes B = \sum_{t=1}^{s}{(M_t \otimes N_t)}$, where the sum is under Boolean arithmetic.

We claim that the families $\mathcal{M} = \{M_{1},...,M_{s}\}$ and $\mathcal{N} = \{N_{1},...,N_{s}\}$ satisfy the assertion of the lemma.
Indeed, the  $i,j$'th block of the matrix $A \otimes B$ is equal to the sum of the matrices $N_t$ for which ${(M_t)}_{i,j} = 1$.
This implies that $\mathcal{M}$ is a cover of $A$, because $M_t \preceq A$ for all $t \in [s]$, and because the blocks of $A \otimes B$ that correspond to $1$-entries of $A$ are equal to $B$.
For the same reason, if $A_{i,j}=1$ then the matrices $N_t$ for which ${(M_t)}_{i,j} = 1$ form a cover of $B$, as required.
\EPF\\

We proceed by presenting a lower bound on the Boolean rank of Kronecker products of $0,1$-matrices.
Here, for a nonzero $0,1$-matrix $A$, we let $\mu(A)$ stand for the number of $1$-entries of $A$ divided by the largest number of $1$-entries in a monochromatic combinatorial rectangle of $A$.

\BT\label{thm:lower}
Let $A$ and $B$ be two nonzero $0,1$-matrices. Then,
\[ \RB(A \otimes B) \geq \mu(A) \cdot \RB(B). \]
\ET
\BPF
For two given nonzero $0,1$-matrices $A$ and $B$, let $s = \RB(A \otimes B)$.
By Lemma~\ref{lemma:lower}, there exist two families of $0,1$-matrices of Boolean rank $1$, $\mathcal{M} = \{M_{1},...,M_{s}\}$ and $\mathcal{N} = \{N_{1},...,N_{s}\}$, satisfying the conditions of the lemma.
For every $i,j$ such that $A_{i,j}=1$, the matrices $N_t$ for which ${(M_t)}_{i,j} = 1$ form a cover of $B$, so,
in particular, the number of these matrices is at least $\RB(B)$. This implies that every $1$-entry of $A$ is covered by at least $\RB(B)$ of the matrices of $\mathcal{M}$. Hence, the total number of ones in the matrices of $\mathcal{M}$ is at least the number of $1$-entries of $A$ times $\RB(B)$. On the other hand, this quantity is clearly bounded from above by $s$ times the largest number of $1$-entries in a monochromatic combinatorial rectangle of $A$. This implies that $s \geq \mu(A) \cdot \RB(B)$, as desired.
\EPF

\begin{remark}
We note that when $A$ is the adjacency matrix of an edge-transitive graph, the quantity $\mu(A)$ coincides with the fractional variant of the Boolean rank (see, e.g.,~\cite{HajiabolhassanM12}).

\end{remark}

As an immediate corollary, we obtain the following.

\begin{coro}\label{cor:lower}
For all integers $n,m \geq 1$, it holds that
\[\RB(C_n \otimes C_m) \geq \frac{n(n-1)}{\lceil{n/2}\rceil \cdot \lfloor{n/2}\rfloor} \cdot \sigma(m).\]
In particular,
\[ \RB(C_n \otimes C_n) \geq (4-o(1)) \cdot \sigma(n). \]
\end{coro}
\BPF
Observe that the matrix $C_n$ has $n(n-1)$ $1$-entries and that the largest number of $1$-entries in a monochromatic combinatorial rectangle of $C_n$ is $\lceil{n/2}\rceil \cdot \lfloor{n/2}\rfloor$. Recall that by Lemma~\ref{lem:R(C_n)} we have $\RB(C_m) = \sigma(m)$, and apply Theorem~\ref{thm:lower} to complete the proof.
\EPF

\section{Explicit Covers of $C_n \otimes C_n$}\label{sec:Covers_Cn}

In this section we present our constructions of covers of the matrix $C_n \otimes C_n$.
To do so, we construct families of matrices that can be used to apply the method presented in the previous section.
We start with some notation that will be used throughout the section.
We then apply a special case of our method, with families of three matrices as in Corollary~\ref{cor:RankAxB_3},
to prove that for all $n \geq 7$ the Boolean rank of $C_n \otimes C_n$ is strictly smaller than the square of the Boolean rank of $C_n$.
Finally, we apply the more general Corollary~\ref{cor:RankAxB_k} to obtain explicit asymptotically nearly optimal covers of $C_n \otimes C_n$.

\subsection{Notation}

For a set $F \subseteq [k]$, denote $\overline{F} = [k] \setminus F$.
Let  $\mathcal{F} = (F_1, \ldots, F_n)$  be an ordered collection of $n$ distinct $\ell$-subsets of $[k]$ for $\ell = \lceil k/2 \rceil$,
and let $\mathcal{\overline{F}} = (\overline{F_1}, \ldots, \overline{F_n})$.
Denote by $B_{\mathcal{F}}$ the $0,1$-matrix of size $n \times n$, whose rows are indexed by the sets of $\mathcal{F} $ and whose columns by the sets
of $\mathcal{\overline{F}}$, such that $(B_{\mathcal{F}})_{i,j} = 1$ if and only if $F_i \cap \overline{F_j} \neq \emptyset$.
Observe that $B_{\mathcal{F}}$ is equal to the matrix $C_n$, because an $\ell$-subset and a $(k-\ell)$-subset of $[k]$ are disjoint if and only if one is the complement of the other.

It is easy to see that $R_{\mathbb{B}}(B_{\mathcal{F}}) \leq k$.
Indeed, for every $t \in [k]$, let $P_{\mathcal{F}}(t)$ be the $0,1$-matrix of size $n \times n$,
whose rows are indexed by the sets of $\mathcal{F} $ and whose columns by the sets of $\mathcal{\overline{F}}$,
such that $(P_{\mathcal{F}}(t))_{i,j}=1$ if and only if $t \in F_i \cap \overline{F_j}$.
Observe that the matrices $P_{\mathcal{F}}(1), \ldots, P_{\mathcal{F}}(k)$ have Boolean rank $1$ and that they form a cover of $B_{\mathcal{F}}$.
We refer to the matrix $P_{\mathcal{F}}(t)$ as the matrix that represents the intersections of the sets of $\mathcal{F}$ and $\mathcal{\overline{F}}$ at $t$.

\subsection{$\RB(C_n \otimes C_n)$ is strictly smaller than $\RB(C_n)^2$}\label{sec:strictly}

We turn to prove an upper bound on $\RB(C_n \otimes C_n)$ using Corollary~\ref{cor:RankAxB_3} and the ideas described briefly in Section~\ref{sec-outline}.
This requires us to construct a triple $A_1,A_2,A_3$ of matrices, such that every two of them cover the matrix $C_n$.
Note that this condition implies a limitation on the Boolean rank of the three matrices, as the sum of the Boolean rank of every two of them must be at least $k$, where $k= \RB(C_n) = \sigma(n)$.

Our strategy for constructing the matrices is the following. Let $\mathcal{F}$ be an ordered collection of all $\ell$-subsets of $[k]$ for $\ell = \lceil k/2 \rceil$.
As explained in the previous subsection, the matrix $B_\mathcal{F}$ is equal to $C_n$ for $n = \binom{k}{\ell}$,
and the $k$ matrices $P_{\mathcal{F}}(1), \ldots, P_{\mathcal{F}}(k)$ have Boolean rank $1$ and form a cover of $C_n$.
Now, for any $r \geq 1$, consider the matrix $A_1$ defined as the sum of the first $r$ matrices from $P_{\mathcal{F}}(1), \ldots, P_{\mathcal{F}}(k)$,
and the matrix $A_2$ defined as the sum of the remaining $k-r$ matrices. This gives us two matrices $A_1$ and $A_2$ of Boolean rank $r$ and $k-r$, respectively, that together cover $C_n$.
We turn now to define a third matrix $A_3$, so that both the pairs $A_1,A_3$ and $A_2,A_3$ cover $C_n$ as well.

As mentioned earlier, if one permutes the sets in $\mathcal{F}$ by some bijection $h: \mathcal{F} \rightarrow \mathcal{F}$ to obtain a different ordered collection $\mathcal{F}'$,
then the matrix $B_{\mathcal{F}'}$ is also equal to $C_n$, whereas the cover associated with it, $P_{\mathcal{F}'}(1), \ldots, P_{\mathcal{F}'}(k)$, is different.
Hence, a possible way to obtain a matrix $A_3$ so that $A_1,A_3$ cover $C_n$, is to consider a bijection $h$ that {\em preserves the intersections} with $[r]$.
This means that the elements of $[r]$ play the same role in $\mathcal{F}$ and $\mathcal{F}'$. Therefore, the first $r$ matrices of the covers corresponding to $\mathcal{F}$ and $\mathcal{F}'$ coincide.
Letting $A_3$ be the sum of the last $k-r$  matrices in the cover that corresponds to $\mathcal{F}'$, guarantees that $A_1,A_3$ form a cover of $C_n$.

We still have to choose $h$ carefully so that $A_2,A_3$ will also cover $C_n$.
Since $A_1,A_2$ cover $C_n$ we have to verify that every entry that is covered by $A_1$ but not by $A_2$, meaning that the sets corresponding to its row and column in $B_{\mathcal{F}}$ intersect only inside $[r]$, will be covered by $A_3$. For this to happen, the bijection $h$ should permute the sets in a way that {\em shifts} every  intersection that fully lies in $[r]$ to an intersection that includes at least one of the other $k-r$ elements.
This essentially reduces the problem of proving an upper bound on the Boolean rank of $C_n \otimes C_n$ to that of finding a bijection $h$ that satisfies the aforementioned properties, defined formally below. See Figure~\ref{C10} for an illustration.

\begin{figure}[htb!]
\captionsetup{width=0.9\textwidth}
\centering
    \includegraphics[width=0.8\textwidth]{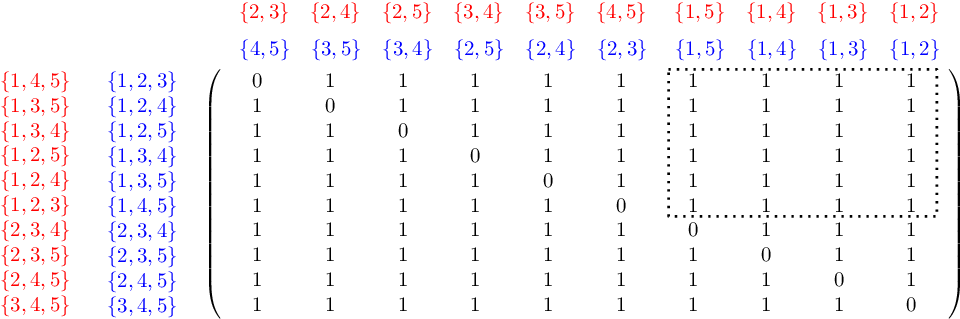}
\caption[center]{\small Two different covers of $C_{10}$. The sets in blue are those assigned by the ordered collection $\mathcal{F}$. The sets in red are assigned by the ordered collection $\mathcal{F}'$ obtained by applying the permutation $h$, defined by $h(F) = \overline{F \setminus \{1\}} \cup \{1\}$ if $1 \in F$, and $h(F) = F$ otherwise. Note that $h$ is $\{1\}$-preserving and $\{1\}$-intersection shifting, and that the highlighted rectangle is common to the two covers associated with $\mathcal{F}$ and $\mathcal{F}'$.}
\label{C10}
\end{figure}

\BD\label{def:h_inter_shift}
Let $\mathcal{F}$ be a family of subsets of $[k]$, let $L$ be a subset of $[k]$, and let $h: \mathcal{F} \rightarrow \mathcal{F}$ be a function.
\begin{enumerate}
\item
The function $h$ is {\em $L$-preserving} if it preserves the
intersections with $L$, that is, for every $D \in \mathcal{F}$, $D \cap L = h(D) \cap L$.
\item
The function $h$ is {\em $L$-intersection shifting} if for every $D,E \in \mathcal{F}$ such that
 $ \emptyset \neq (D \cap \overline{E}) \subseteq L$ it holds that $ (h(D) \cap \overline{h(E)}) \setminus L \neq \emptyset.$
\end{enumerate}
\ED

The following theorem formally shows how bijections that satisfy the properties given in Definition~\ref{def:h_inter_shift} can be used to obtain upper bounds on the Boolean rank of $C_n \otimes C_n$.

\BT\label{thm:Cn_general}
Let $k \geq r \geq 1$ be integers, and let $\mathcal{F}$ be a family of $n$ distinct $\ell$-subsets of $[k]$ for $\ell = \lceil k/2 \rceil$.
Suppose that there exists an $[r]$-preserving and $[r]$-intersection shifting bijection $h: \mathcal{F} \rightarrow \mathcal{F}$.
Then, the matrix $C_n$ is $(r,k-r,k-r)$-coverable. In particular, $R_{\mathbb{B}}(C_n \otimes C_n) \leq k^2-r^2$.
\ET

\BPF
We refer to the family $\mathcal{F}$ as an ordered collection $\mathcal{F} = (F_1, \ldots, F_n)$ of $n$ distinct $\ell$-subsets of $[k]$, where the order is arbitrary.
Recall that the matrix $B_{\mathcal{F}}$ is covered by the $k$ matrices $P_{\mathcal{F}}(1), \ldots, P_{\mathcal{F}}(k)$,
where $P_{\mathcal{F}}(t)$ is the matrix of Boolean rank $1$ that represents the intersections of the sets of $\mathcal{F}$ and $\overline{\mathcal{F}}$ at $t$.

Let $\mathcal{F}'$ be the ordered collection of the sets of $\mathcal{F}$ obtained by applying to every set the given bijection $h$.
Again, the matrix $B_{\mathcal{F}'}$ is covered by the $k$ matrices $P_{\mathcal{F}'}(1), \ldots, P_{\mathcal{F}'}(k)$.
Observe that both the matrices $B_{\mathcal{F}}$ and $B_{\mathcal{F}'}$ are in fact equal to $C_n$.
Since $h$ is $[r]$-preserving, every $t \in [r]$ plays the exact same role in the ordered collections $\mathcal{F}$ and $\mathcal{F}'$. Hence, $P_{\mathcal{F}}(t) = P_{\mathcal{F}'}(t)$ for every such $t$.

To prove that $C_n$ is $(r,k-r,k-r)$-coverable, consider the following three matrices:
\[A_1 = \sum_{t \in [r]}{P_{\mathcal{F}}(t)},~~~ A_2 = \sum_{t \in [k] \setminus [r]}{P_{\mathcal{F}}(t)},~~~ \mbox{~and~}~~ A_3 = \sum_{t \in [k] \setminus [r]}{P_{\mathcal{F}'}(t)},\]
where the sum is with respect to Boolean arithmetic.
It holds that $R_{\mathbb{B}}(A_1) \leq r$, $R_{\mathbb{B}}(A_2) \leq k-r$, and $R_{\mathbb{B}}(A_3) \leq k-r$, so it remains to show that every two of the matrices $A_1,A_2,A_3$ cover the matrix $C_n$.
By
\[A_1 = \sum_{t \in [r]}{P_{\mathcal{F}}(t)} = \sum_{t \in [r]}{P_{\mathcal{F}'}(t)},\]
each of the pairs $A_1,A_2$ and $A_1,A_3$ obviously covers $C_n$.

To prove that the pair $A_2,A_3$ covers $C_n$ as well, we show that every nonzero entry of $C_n$ that is uncovered by $A_2$ is covered by $A_3$.
So consider such an entry of $C_n$.
Let $D$ and $\overline{E}$, respectively, be the subsets that correspond to the row and column of this entry in the matrix $B_{\mathcal{F}}$.
Notice that the subsets $h(D)$ and $\overline{h(E)}$ correspond to the row and column of this entry in the matrix $B_{\mathcal{F}'}$.
Since our entry is uncovered by $A_2$, we have $\emptyset \neq (D \cap \overline{E}) \subseteq [r]$.
Since $h$ is $[r]$-intersection shifting, it follows that $(h(D) \cap \overline{h(E)}) \setminus [r] \neq \emptyset$, hence our entry is covered by $A_3$.
It follows that $C_n$ is $(r,k-r,k-r)$-coverable, as required.

Finally, the fact that $C_n$ is $(r,k-r,k-r)$-coverable implies that it is also $(k-r,r,k-r)$-coverable. By Corollary~\ref{cor:RankAxB_3}, we obtain that
\[ R_{\mathbb{B}}(C_n \otimes C_n) \leq r \cdot (k-r) + (k-r) \cdot r +(k-r)^2 = k^2-r^2,\]
and we are done.
\EPF\\

In order to prove that the Boolean rank of $C_n \otimes C_n$ is strictly smaller than $R_{\mathbb{B}}(C_n)^2$ for all $n \geq 7$, we apply Theorem~\ref{thm:Cn_general} with $r=1$.
The following two simple lemmas are used to obtain the bijection $h$ needed for applying the theorem.

\BL\label{lem:func_g}
Let $k \geq 2$ be an integer and denote $\ell = \lceil k/2 \rceil$.
Then, there exists a bijection $g$ from the family of all $(\ell-1)$-subsets of $[k-1]$ to itself satisfying $F \cap g(F) = \emptyset$ for every set $F$.
\EL

\BPF
If $k$ is odd then $k = 2\ell-1$, and the function $g$ defined by $g(F) = [k-1] \setminus F$ clearly satisfies the assertion of the lemma.

Otherwise, if $k$ is even, we have $k = 2\ell$.
Consider the bipartite graph whose vertices on one side are all the $(\ell-1)$-subsets of $[k-1]$,
and the vertices on the other side are all the $\ell$-subsets of $[k-1]$.
Connect by an edge an $(\ell-1)$-subset and an $\ell$-subset if the former is contained in the latter.
This bipartite graph is $\ell$-regular, and thus, by Hall's marriage theorem it has a perfect matching.

Now, for any $(\ell-1)$-subset $F$ of $[k-1]$, let $\widetilde{F}$ denote the $\ell$-subset of $[k-1]$ adjacent to $F$ in this prefect matching, and define $g(F) = [k-1] \setminus \widetilde{F}$.
The function $g$ is a bijection because it is a composition of two bijections (one defined by the prefect matching and one by the complement operation).
Moreover, for every $(\ell-1)$-subset $F$ of $[k-1]$ it holds that $F \subseteq \widetilde{F}$, and thus $F \cap g(F) = \emptyset$, as desired.
\EPF

\BL\label{lem:func_h_n>=7}
Let $k \geq 5$ be an integer, let $i \in [k]$, denote $\ell = \lceil k/2 \rceil$, and let $\mathcal{F}$ be the family of all $\ell$-subsets of $[k]$.
Then, there exists an $\{i\}$-preserving and $\{i\}$-intersection shifting bijection $h: \mathcal{F} \rightarrow \mathcal{F}$.
\EL

\BPF
By Lemma~\ref{lem:func_g}, there exists a bijection $g$ from the family of all $(\ell-1)$-subsets of $[k] \setminus \{i\}$ to itself satisfying $F \cap g(F) = \emptyset$ for every set $F$.
Let $h: \mathcal{F} \rightarrow \mathcal{F}$ be the function defined by
\[ h(F) = \left\{ \begin{array}{ll}
          g(F \setminus \{i\}) \cup \{i\}, & \mbox{if $i \in F$}.\\
       F, & \mbox{if $i \notin F $}.\end{array} \right. \]
It can easily be seen that $h$ is an $\{i\}$-preserving bijection from $\mathcal{F}$ to itself.

To prove that $h$ is $\{i\}$-intersection shifting, suppose that $D \cap \overline{E} = \{i\}$ for some $D,E \in \mathcal{F}$,
let $D' = D \setminus \{i\}$, and note that $D' \cap \overline{E} = \emptyset$.
By definition, $h(D) = g(D') \cup \{i\}$ and $h(E) = E$, and our choice of $g$ implies that $D' \cap g(D') = \emptyset$.
However, the three sets $D'$, $g(D')$ and $\overline{E}$ cannot be pairwise disjoint, because
\[|D'|+|g(D')|+|\overline{E}| = (\ell-1)+(\ell-1)+(k-\ell) = k+\ell-2 = k+\lceil k/2 \rceil -2 > k,\]
where for the inequality we use the assumption that $k \geq 5$.
This implies that $g(D') \cap \overline{E} \neq \emptyset$, and thus, $(h(D) \cap \overline{E}) \setminus \{i\} \neq \emptyset$, as required.
\EPF\\

Equipped with Theorem~\ref{thm:Cn_general} and Lemma~\ref{lem:func_h_n>=7}, we are ready to prove the following.

\BT\label{thm:1gap}
For every integer $n \geq 7$, the matrix $C_n$ is $(k-1,k-1,1)$-coverable for $k = \sigma(n)$.
In particular, for all integers $n,m \geq 7$, $R_{\mathbb{B}}(C_n \otimes C_m) < R_{\mathbb{B}}(C_n) \cdot \RB(C_m)$.
\ET

\BPF
For any given $k \geq 5$, let $\ell = \lceil k/2 \rceil$ and $n = \binom{k}{ \ell }$.
By Lemma~\ref{lem:R(C_n)}, $R_{\mathbb{B}}(C_n) = \sigma(n)=k$.
It suffices to prove the assertion of the theorem for such values of $n$, because every matrix $C_{n'}$
with $R_{\mathbb{B}}(C_{n'})=\sigma(n')=k$ is a sub-matrix of $C_n$, and because $\sigma(n') \geq 5$ whenever $n' \geq 7$.

Consider the collection $\mathcal{F}$ of all $\ell$-subsets of $[k]$.
By Lemma~\ref{lem:func_h_n>=7}, there exists a $\{1\}$-preserving and $\{1\}$-intersection shifting bijection $h: \mathcal{F} \rightarrow \mathcal{F}$.
Applying Theorem~\ref{thm:Cn_general} with $r=1$ and this $h$, we obtain that the matrix $C_n$ is $(1,k-1,k-1)$-coverable, as required.

Finally, for any integers $n,m \geq 7$, the matrix $C_n$ is $(1,k_1-1,k_1-1)$-coverable for $k_1 = \sigma(n)$ and the matrix $C_m$ is $(k_2-1,k_2-1,1)$-coverable for $k_2 = \sigma(m)$. Applying Corollary~\ref{cor:RankAxB_3}, we obtain that
\begin{eqnarray*}
R_{\mathbb{B}}(C_n \otimes C_m) &\leq& 1\cdot(k_2-1) + (k_1-1)(k_2-1) + (k_1-1)\cdot 1 = k_1 \cdot k_2 -1 \\
&=& \RB(C_n) \cdot \RB(C_m)-1 < \RB(C_n) \cdot \RB(C_m),
\end{eqnarray*}
and we are done.
\EPF\\

\subsubsection{Additional Examples}\label{subsub:examples}

We remark that Theorem~\ref{thm:1gap} does not provide useful covers of $C_n$ for $n \in \{5,6\}$.
While it follows from~\cite{watts2001boolean} that the matrix $C_4$ is $(2,2,2)$-coverable, it is not difficult to verify, using ideas similar to those of Theorem~\ref{thm:lower}, that $C_5$ is not $(2,2,2)$-coverable nor $(3,3,1)$-coverable.
Yet, it turns out that $C_5$ is $(2,2,3)$-coverable. To see this, assign to its rows and columns subsets of $[7]$ as follows.
\begin{equation*}
\kbordermatrix{                &   \{2,3,6\} &   \{2,4,5\} &   \{1,3,5\} &  \{1,4,6\} &   \{1,2,7\}  \\
                             \{1,4,5,7\}     & 0 & 1 & 1 & 1 & 1  \\
                             \{1,3,6,7\}     & 1 & 0 & 1 & 1 & 1  \\
                             \{2,4,6,7\}     & 1 & 1 & 0 & 1 & 1  \\
                             \{2,3,5,7\}     & 1 & 1 & 1 & 0 & 1  \\
                             \{3,4,5,6\}     & 1 & 1 & 1 & 1 & 0  }
\end{equation*}
Now, let $A_1$ be the matrix that represents the intersections of these sets at $1$ and $2$, let $A_2$ be the matrix that represents the intersections at $3$ and $4$, and let $A_3$ be the matrix that represents the intersections at $5$, $6$ and $7$. This gives us the following three matrices.
\begin{equation*}
A_1 = \left(
        \begin{array}{ccccc}
          0 & 0 & 1 & 1 & 1 \\
          0 & 0 & 1 & 1 & 1 \\
          1 & 1 & 0 & 0 & 1 \\
          1 & 1 & 0 & 0 & 1 \\
          0 & 0 & 0 & 0 & 0 \\
        \end{array}
      \right),~~~~
A_2 = \left(
        \begin{array}{ccccc}
          0 & 1 & 0 & 1 & 0 \\
          1 & 0 & 1 & 0 & 0 \\
          0 & 1 & 0 & 1 & 0 \\
          1 & 0 & 1 & 0 & 0 \\
          1 & 1 & 1 & 1 & 0 \\
        \end{array}
      \right),~~~~
A_3 = \left(
        \begin{array}{ccccc}
          0 & 1 & 1 & 0 & 1 \\
          1 & 0 & 0 & 1 & 1 \\
          1 & 0 & 0 & 1 & 1 \\
          0 & 1 & 1 & 0 & 1 \\
          1 & 1 & 1 & 1 & 0 \\
        \end{array}
      \right).
\end{equation*}
It is easy to verify that every two of them cover $C_5$, implying that $C_5$ is $(2,2,3)$-coverable.
Combining this with Corollary~\ref{cor:RankAxB_3} and Theorem~\ref{thm:1gap}, we get that for every $n \in \{4,5\}$ and $m \geq 7$,
it holds that $R_{\mathbb{B}}(C_n \otimes C_m) < R_{\mathbb{B}}(C_n) \cdot \RB(C_m)$, and that $R_{\mathbb{B}}(C_4 \otimes C_5) \leq 14$.
By Theorem~\ref{thm:lower}, the latter upper bound is tight, namely, $R_{\mathbb{B}}(C_4 \otimes C_5) = 14$.

We further remark that our approach is not limited to the matrices $C_n$ and might be useful for other matrices as well.
As an example, consider the $7 \times 7$ matrix $P_7$ that has zeros on the main diagonal and on the diagonal above it and ones everywhere else.
While $\RB(P_7)=6$ (see~\cite{CaenBicliques}), it can be shown that $P_7$ is $(5,5,1)$-coverable (see Figure~\ref{P7}).
By Corollary~\ref{cor:RankAxB_3}, this implies that $\RB(P_7 \otimes P_7) \leq 35 < \RB(P_7)^2$.

\begin{figure}[htb!]
\captionsetup{width=0.9\textwidth}
\centering
    \includegraphics[width=0.6\textwidth]{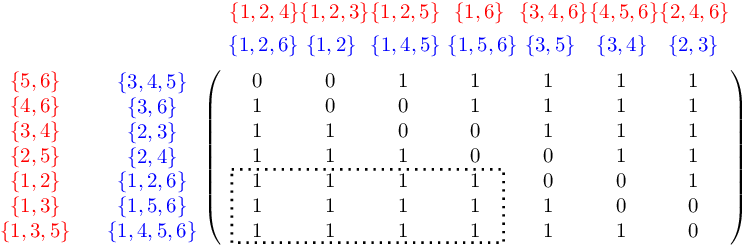}
\caption[center]{\small Two different covers of $P_{7}$ with a common rectangle.
Let $A_1$ be the matrix that represents the intersections of the red subsets assigned to the rows and columns of $P_7$ at $2,3,4,5,6$,
let $A_2$ be the matrix that represents the intersections of the blue subsets at $2,3,4,5,6$,
and let $A_3$ be the matrix that represents the intersections of the subsets at $1$.
Then each of the pairs $A_1,A_2$, $A_1, A_3$ and $A_2,A_3$ form a cover of $P_7$, and hence $P_7$ is $(5,5,1)$-coverable.}
\label{P7}
\end{figure}

\subsection{An Explicit Cover of $C_n \otimes C_n$}\label{sec:asymptotic}

We turn to provide an explicit construction of a cover for the matrix $C_n \otimes C_n$ of size nearly linear in the Boolean rank of $C_n$ (up to a $\log_2 \log_2 n$ multiplicative term).
We again mention that it is also possible to explicitly construct a cover of $C_n \otimes C_n$ without the $\log_2 \log_2 n$ loss in the bound~\cite{KushilevitzPersonal}.
For our construction, we need Corollary~\ref{cor:RankAxB_k} in its full generality rather than Corollary~\ref{cor:RankAxB_3} that was used in the previous subsections.
Our goal, then, is to construct a family of $s$ matrices such that every $\lceil s/2 \rceil$ of them cover $C_n$.
Letting $k =\RB(C_n)$, we would like the Boolean rank of every matrix in the collection to be close to $2k/s$, which is the minimum Boolean rank of the matrices we can hope for, given that every $\lceil s/2 \rceil$ of them cover $C_n$.

As before, we start with some ordered collection $\mathcal{F}$ of $\ell$-subsets of $[k]$ for $\ell = \lceil k/2 \rceil$, for which we consider the natural cover of $B_\mathcal{F}$ by the $k$ matrices $P_{\mathcal{F}}(1), \ldots, P_{\mathcal{F}}(k)$ of Boolean rank $1$. For any integer $r \geq 1$, one can associate with $\mathcal{F}$ the matrix $A = \sum_{t \in [r]}{P_{\mathcal{F}}(t)}$. Such a matrix can also be associated with every permutation $\mathcal{F}'$ of the sets of $\mathcal{F}$.
Our goal is to construct a family of $s$ permutations of $\mathcal{F}$ so that the matrices associated with every $\lceil s/2 \rceil$ of them cover $C_n$.

Our construction relies on the following lemma, which provides a family of functions from a set of size $p^q$ to a set of size $p$, where $p$ is a prime and $q$ is some integer.
For $q \geq 2$, these functions obviously cannot be injective and must have collisions, i.e., pairs of inputs that are mapped to the same output.
A crucial property that we need is that every $q$ of the functions do not share a common collision. This is obtained by a simple algebraic argument described below.

\BL\label{lem:g's_prime_new}
Let $q$ be an integer.
For every $q$ sets $U_1,\ldots,U_q$ of prime size $p > q$, there exists an explicit construction of $p-1$ functions
\[g_i:U_1 \times \cdots \times U_q \rightarrow U_1~~~~~ (i \in [p-1])\]
satisfying that
\begin{enumerate}
  \item for every $i \in [p-1]$ and every fixed $(x_2,\ldots,x_q) \in U_2 \times \cdots \times U_{q}$, the restriction of $g_i$ to the $q$-tuples $(x_1,x_2 \ldots, x_q)$ with $x_1 \in U_1$ is a bijection onto $U_1$, and
  \item for every distinct $i_1,\ldots,i_q \in [p-1]$ and every two $q$-tuples $x,y \in U_1 \times \cdots \times U_{q}$,
  \[\mbox{if~} g_{i_j}(x) = g_{i_j}(y) \mbox{~~for all~~} j \in [q] \mbox{~~then~~} x=y.\]
\end{enumerate}
\EL

\BPF
It suffices to prove the lemma for $U_1 = \cdots = U_q = \mathbb{Z}_p$.
For every $i \in \mathbb{Z}_p \setminus \{0\}$ consider the function $g_i:\mathbb{Z}^q_p \rightarrow \mathbb{Z}_p$ defined by
\[g_i(x) = \sum_{t=1}^{q}{i^{t-1} \cdot x_t}\]
for every $x\in \mathbb{Z}^q_p$.
It is easy to see that for every fixed $(x_2,\ldots,x_q) \in \mathbb{Z}_{p}^{q-1}$, the restriction of $g_i$ to the $p$ $q$-tuples $(x_1,x_2 \ldots, x_q)$ with $x_1 \in \mathbb{Z}_p$ is a linear function in $x_1$, where the coefficient of $x_1$ is $1$. Hence, it is a bijection from $\mathbb{Z}_p$ to itself.

To prove the other required property, take distinct $i_1,\ldots,i_q \in [p-1]$ and $x,y \in \mathbb{Z}^{q}_p$, and suppose that $g_{i_j}(x) = g_{i_j}(y)$ for all $j \in [q]$. This implies that for every $j \in [q]$ it holds that
\[ \sum_{t=1}^{q}{{i_j}^{t-1} \cdot x_t} = \sum_{t=1}^{q}{{i_j}^{t-1} \cdot y_t},\]
which is equivalent to
\[ \sum_{t=1}^{q}{{i_j}^{t-1} \cdot (x_t-y_t)} = 0.\]
The matrix corresponding to this system of linear equations, with variables $x_t-y_t$, is precisely the Vandermonde matrix associated with the values $i_1, \ldots, i_q$.
Since they are distinct, the matrix is invertible and the zero vector is the unique solution of the system.
This implies that $x=y$, completing the proof.
\EPF\\

The functions given by Lemma~\ref{lem:g's_prime_new} are used to construct a family of $0,1$-matrices such that every $q$ of them cover the matrix $C_n$.

\BL\label{lemma:gen_l_cover}
For a sufficiently large integer $d$ and for an integer $q \geq 2$, let
\[n = \bigg (\frac{1}{2} \cdot \binom{2d}{d} \bigg )^q.\]
Then, there exists an explicit construction of $n^{1/q}-1$  matrices of size $n \times n$  and of Boolean rank at most $2d$, such that every $q$ of them cover $C_n$.
\EL

\BPF
Given a sufficiently large integer $d$ and an integer $q \geq 2$, let $p$ be the largest prime number satisfying $p \leq \binom{2d}{d}$.
It is well known that for every sufficiently large integer $m$, the interval $[m,2m]$ includes a prime number, hence $p \geq \frac{1}{2} \cdot \binom{2d}{d}$.
We turn to present an explicit construction of $p-1$   matrices of size $p^{q} \times p^{q}$ and of Boolean rank at most $2d$,
such that every $q$ of them cover the matrix $C_{p^{q}}$. This will imply the assertion of the lemma.

Let $k = 2dq$, and let $[k] = L_1 \cup \cdots \cup L_q$ be a partition of $[k]$ into $q$ disjoint sets of size $2d$ each.
For every $j \in [q]$, let $\mathcal{B}_j$ be a collection of $p$ arbitrarily chosen $d$-subsets of $L_j$.
Let $\mathcal{F}$ be an arbitrary ordering of all subsets $B$ of $[k]$ satisfying $B \cap L_j \in \mathcal{B}_j$ for all $j \in [q]$.
Note that the size of every set in $\mathcal{F}$ is $dq=k/2$, and that the number of sets in $\mathcal{F}$ is $\prod_{j \in [q]}{|\mathcal{B}_j|} = p^q$.

Since $|\mathcal{B}_j|=p$ for all $j \in [q]$, Lemma~\ref{lem:g's_prime_new} provides us with $p-1$ functions
\[g_i:\mathcal{B}_1 \times \cdots \times \mathcal{B}_q \rightarrow \mathcal{B}_1~~~~~ (i \in [p-1])\]
satisfying that
\begin{enumerate}
  \item\label{property1_p} for every $i \in [p-1]$ and every fixed $(x_2,\ldots,x_q) \in \mathcal{B}_2 \times \cdots \times \mathcal{B}_{q}$, the restriction of $g_i$ to the $q$-tuples $(x_1,x_2 \ldots, x_q)$ with $x_1 \in \mathcal{B}_1$ is a bijection onto $\mathcal{B}_1$, and
  \item\label{property2_p} for every distinct $i_1,\ldots,i_q \in [p-1]$ and every two $q$-tuples $x,y \in \mathcal{B}_1 \times \cdots \times \mathcal{B}_{q}$,
  \[\mbox{if~} g_{i_j}(x) = g_{i_j}(y) \mbox{~~for all~~} j \in [q] \mbox{~~then~~} x=y.\]
\end{enumerate}
For every $i \in [p-1]$ we define a bijection $h_i: \mathcal{F} \rightarrow \mathcal{F}$ as follows.
For a subset $D \in \mathcal{F}$, let $D = D_1 \cup \cdots \cup D_q$ where $D_j = D \cap L_j$ ($j \in [q]$), and define
\[h_i(D) = g_i(D_1,\ldots,D_q) \cup D_2 \cdots \cup D_{q}.\]
Observe that for a given subset $D \in \mathcal{F}$, the function $h_i$ does not change the elements of $D$ outside of $L_1$.
The elements of $D$ in $L_1$, however, are changed according to the function $g_i$ applied to $(D_1, \ldots,D_q)$.
It follows from Item~\ref{property1_p} that every $h_i$ is a bijection from $\mathcal{F}$ to itself.
Let $\mathcal{F}_i$ be the ordered collection obtained from $\mathcal{F}$ by applying the bijection $h_i$ to its sets.

Now, for every $i \in [p-1]$, consider the matrix
\[ A_i = \sum_{t \in L_1}{P_{\mathcal{F}_i}(t)}.\]
The Boolean rank of each of these $p-1$ matrices is clearly at most $|L_1| = 2d$.
We turn to show that every $q$ of them cover $C_{p^q}$.
To do so, consider an arbitrary entry of the matrix $C_{p^q}$.
Let $D$ and $\overline{E}$ be the subsets that correspond, respectively, to the row and column of this entry, where $C_{p^q}$ is viewed as $B_\mathcal{F}$.
Let $D = D_1 \cup \cdots \cup D_q$ and $E = E_1 \cup \cdots \cup E_q$, where $D_j = D \cap L_j$ and $E_j = E \cap L_j$ for $j \in [q]$.
Suppose that for some distinct $i_1,\ldots,i_q \in [p-1]$ this entry is not covered by $A_{i_j}$ for all $j \in [q]$.
This means that for all $j \in [q]$, it holds that
\[h_{i_j}(D) \cap \overline{h_{i_j}(E)} \cap L_1 = \emptyset,\]
equivalently,
\[g_{i_j}(D_1,\ldots,D_q) \cap (L_1 \setminus g_{i_j}(E_1,\ldots,E_q)) = \emptyset.\]
Since all the sets in $\mathcal{B}_1$ are of size precisely $d$, it follows that $g_{i_j}(D_1 ,\ldots, D_q) = g_{i_j}(E_1,\ldots,E_q)$ for all $j \in [q]$, hence, Item~\ref{property2_p} implies that $D=E$.
This means that for every distinct $i_1,\ldots,i_q \in [p-1]$ the only entries that are uncovered by the matrices $A_{i_j}$ with $j \in [q]$ are those whose row and column correspond to complement sets.
Hence, they cover $C_{p^q}$ and we are done.
\EPF\\

We are now ready to obtain our explicit nearly optimal construction of a cover for $C_n \otimes C_n$.
Lemma~\ref{lemma:gen_l_cover} provides a collection of matrices such that every $q$ of them cover $C_n$.
To fit the condition of Corollary~\ref{cor:RankAxB_k} we need at least $2q$ such matrices.
This leads us to the setting of parameters with $q = \Theta(k/\log k)$, where $k = \sigma(n)$, as shown below.

\BT\label{thm:mlogm}
There exists a constant $c>0$ for which the following holds.
Given an integer $n$, it is possible to construct deterministically in running-time polynomial in $n$, a cover of the matrix $C_n \otimes C_n$
with at most $c \cdot k \log_2 k$ monochromatic rectangles, where $k = \RB(C_n)$.
\ET

\BPF
We first consider integers $n$ of the form $n = \binom{k}{k/2}$ where $k$ is even.
By Lemma~\ref{lem:R(C_n)}, we have $\RB(C_n) = k$.
Let $d$ be the smallest integer satisfying $n \leq n_d$, where
\[ n_d = \bigg ( \frac{1}{2} \cdot \binom{2d}{d} \bigg )^{2^d}.\]
Note that it can be assumed that the integer $d$, just like the integer $n$, is sufficiently large.
By the minimality of $d$, it follows that $n > n_{d-1}$, hence
\[2^k \geq n > n_{d-1} = \bigg ( \frac{1}{2} \cdot \binom{2(d-1)}{d-1} \bigg )^{2^{d-1}} \geq \bigg ( \frac{2^{2d-4}}{\sqrt{d-1}} \bigg )^{2^{d-1}},\]
where the last inequality holds by Fact~\ref{fact:central_binom}. It thus follows that
\begin{eqnarray}\label{eq:k_vs_d}
k \geq 2^{d-1} \cdot \Big ( 2d- 4-\tfrac{1}{2} \cdot \log_2 (d-1) \Big ) = \Theta(2^d \cdot d).
\end{eqnarray}

Now, apply Lemma~\ref{lemma:gen_l_cover} with $q = 2^d$ to obtain a collection of $0,1$-matrices of size $n_d \times n_d$ and of Boolean rank at most $2d$,
such that every $2^d$ of them cover $C_{n_d}$. The number of matrices in this collection, according to the lemma, is
\[ n_d^{1/2^d}-1 = \frac{1}{2} \cdot \binom{2d}{d} - 1 \geq 2^{d+1},\]
where for the inequality we again use the fact that $d$ is sufficiently large.
Applying Corollary~\ref{cor:RankAxB_k} to $2^{d+1}$ of these matrices, every $2^d$ of which cover $C_{n_d}$, we obtain that
\[R_{\mathbb{B}}(C_{n_d} \otimes C_{n_d}) \leq 2^{d+1} \cdot (2d)^2 = 2^{d+3} \cdot d^2.\]
Since $n \leq n_d$, the matrix $C_n$ is a sub-matrix of $C_{n_d}$, implying that
\[R_{\mathbb{B}}(C_n \otimes C_n) \leq R_{\mathbb{B}}(C_{n_d} \otimes C_{n_d}) \leq 2^{d+3} \cdot d^2 \leq O(k \cdot  \log_2 k),\]
where the last inequality holds by~\eqref{eq:k_vs_d}.

We next observe that if the bound holds for integers $n$ of the form $n = \binom{k}{k/2}$, where $k$ is even, then it holds in general.
Indeed, for every $n'$ with even $k=\sigma(n')$, $C_{n'}$ is a sub-matrix of $C_n$ for $n = \binom{k}{k/2}$ and it holds that $\sigma(n)=\sigma(n')$.
Further, for $n'$ with odd $k=\sigma(n')$, $C_{n'}$ is a sub-matrix of $C_n$ for $n = \binom{k+1}{ (k+1)/2}$ and $\sigma(n)=\sigma(n')+1$. So for an appropriate constant $c$ in the theorem, the statement holds for such $n'$ as well.

We finally mention that one can observe that the proof provides not only a bound on the Boolean rank of $C_n \otimes C_n$, but also an explicit construction of a cover attaining it. To see this, recall that Lemma~\ref{lemma:gen_l_cover} explicitly defines the matrices used above, based on the algebraic argument of Lemma~\ref{lem:g's_prime_new}. Moreover, the construction provides a presentation of those matrices as a sum of matrices of Boolean rank $1$. As follows from the proof of Theorem~\ref{theorem:RankTensorIntro}, products of these matrices (of Boolean rank $1$ each) can be used to obtain the required cover of $C_n \otimes C_n$ (see Remark~\ref{remark:explicitThm1}).
This implies that the monochromatic combinatorial rectangles attaining the above bound can be computed deterministically in running-time polynomial in $n$, as required.
\EPF

\section{The Rank of Products of Spanoids}\label{sec:spanoids}

A recent work of Dvir, Gopi, Gu, and Wigderson~\cite{DvirGGW20} introduces the concept of {\em spanoids} that captures various mathematical objects of interest.
In this section we show that the approach taken in the current work for proving upper bounds on the Boolean rank of Kronecker products of matrices can be generalized to proving upper bounds on the rank of products of spanoids. We start with some basic definitions. For an in-depth introduction to the topic, the reader is referred to~\cite{DvirGGW20}.

\subsection{Definitions}
A {\em spanoid} $\mathcal{S}$ over a finite set $U$ is a family of pairs $(S,i)$ with $S \subseteq U$ and $i \in U$.
A pair $(S,i) \in \mathcal{S}$ is referred to as an {\em inference rule} and is read as ``$S$ spans $i$''.
It can be assumed that for all $i \in U$, the pair $(\{i\},i)$ is in $\mathcal{S}$ and that monotonicity holds, that is, if $(S,i) \in \mathcal{S}$ then $(S',i) \in \mathcal{S}$ whenever $S \subseteq S' \subseteq U$.
We note that the representation of a spanoid as a collection of pairs is not unique.

A {\em derivation} in $\mathcal{S}$ of $i \in U$ from $T \subseteq U$, written $T \models_\mathcal{S} i$, is a sequence of sets $T_0,T_1,\ldots,T_r \subseteq U$, where $T=T_0$ and $i \in T_r$,
such that for each $j \in [r]$, $T_j = T_{j-1} \cup \{i_j\}$ for some $i_j \in U$
and there exists $S \subseteq T_{j-1}$ for which $(S,i_j) \in \mathcal{S}$.
When the spanoid $\mathcal{S}$ is clear from the context, we simply write $T \models i$.

The {\em span} of a set $T \subseteq U$ in the spanoid $\mathcal{S}$, denoted $\span(T)$, is the set of all $i \in U$ for which $T \models_\mathcal{S} i$.
Finally, the {\em rank} of $\mathcal{S}$, denoted $R(\mathcal{S})$, is the smallest possible size of a set $S \subseteq U$ that spans $U$, that is, satisfies $\span(S)=U$.

\subsection{Boolean Rank as a Spanoid Rank}
We observe that the Boolean rank of $0,1$-matrices can be represented in the language of spanoids.
To see this, consider the following definition.

\BD\label{def:S_A}
For a $0,1$-matrix $A$, let $U_A$ be the collection of all $0,1$-matrices $M$ of Boolean rank $1$ satisfying $M \preceq A$.
Note that $U_A$ in particular includes all the matrices that have a single $1$ that corresponds to a nonzero entry of $A$.

The {\em spanoid associated with a matrix $A$}, denoted $\mathcal{S}_A$, is the spanoid over the set $U_A$,
where a collection $S \subseteq U_A$ of matrices spans a matrix $M \in U_A$ if $M \preceq \sum_{M' \in S}{M'}$ under Boolean arithmetic,
that is, every nonzero entry of $M$ is also nonzero in at least one of the matrices of $S$.
\ED
It is easy to verify that the rank $\R(\mathcal{S}_A)$ of the spanoid $\mathcal{S}_A$ associated with a $0,1$-matrix $A$ coincides with $R_{\mathbb{B}}(A)$, the Boolean rank of the matrix $A$.

\begin{propo}
For every $0,1$-matrix $A$, $R(\mathcal{S}_A) = R_{\mathbb{B}}(A)$.
\end{propo}

\subsection{Product of Spanoids}
Several product operations of spanoids were defined in~\cite{DvirGGW20}. Here we consider the one inspired by the Kronecker product operation on matrices.
\BD\label{def:S_1xS_2}
Let $\mathcal{S}_1$ and $\mathcal{S}_2$ be two spanoids on the sets $U_1$ and $U_2$ respectively.
The {\em product} $\mathcal{S}_1 \otimes \mathcal{S}_2$ is a spanoid on $U_1 \times U_2$ defined by the following inference rules:
\begin{enumerate}
  \item For $S \subseteq U_1$ and $i \in U_1$, if $(S,i) \in \mathcal{S}_1$ then for every $j \in U_2$, $(S \times \{j\}, (i,j)) \in \mathcal{S}_1 \otimes \mathcal{S}_2$.
  \item For $S \subseteq U_2$ and $j \in U_2$, if $(S,j) \in \mathcal{S}_2$ then for every $i \in U_1$, $(\{i\} \times S, (i,j)) \in \mathcal{S}_1 \otimes \mathcal{S}_2$.
\end{enumerate}
\ED
It can be seen that the rank of spanoids is sub-multiplicative with respect to this notion of product, that is, for every two spanoids $\mathcal{S}_1$ and $\mathcal{S}_2$, it holds that $R(\mathcal{S}_1 \otimes \mathcal{S}_2) \leq R(\mathcal{S}_1) \cdot R(\mathcal{S}_2)$. This inequality is sometimes strict, as demonstrated by two examples given in~\cite[Example~6.15]{DvirGGW20}.

We observe below that the rank of the product of spanoids associated with $0,1$-matrices, as in Definition~\ref{def:S_A}, is equal to the Boolean rank of the Kronecker product of the matrices.
This in particular implies that all the explicit gaps obtained in the current work between the Boolean rank of the Kronecker product of two matrices and the product of the Boolean rank of each of them, can also be formulated as explicit gaps for the analogue question for spanoids.

\begin{propo}\label{prop:R_bool_vs_R}
For every two $0,1$-matrices $A$ and $B$, $R_{\mathbb{B}}(A \otimes B) = R(\mathcal{S}_A \otimes \mathcal{S}_B)$.
\end{propo}

\begin{proof}
We first show that $R_{\mathbb{B}}(A \otimes B) \leq R(\mathcal{S}_A \otimes \mathcal{S}_B)$.
Let $t = R(\mathcal{S}_A \otimes \mathcal{S}_B)$. Then there exists a set $S = \{(M_1,N_1), \ldots, (M_t,N_t)\}$ of $t$ pairs spanning the set $U_A \times U_B$ in the spanoid $\mathcal{S}_A \otimes \mathcal{S}_B$.
In particular, for every two matrices $D_A$ and $D_B$ with precisely one nonzero entry such that $D_A \preceq A$ and $D_B \preceq B$, it holds that $S \models (D_A,D_B)$ in $\mathcal{S}_1 \otimes \mathcal{S}_2$.
It suffices to show that for every such $D_A$ and $D_B$, the nonzero entry of $D_A \otimes D_B$ is covered by at least one of the $t$ matrices $M_i \otimes N_i$ with $i \in [t]$, because the Boolean rank of each of them is $1$, implying that $R_{\mathbb{B}}(A \otimes B) \leq t$.

To see this, observe, by combining Definitions~\ref{def:S_A} and~\ref{def:S_1xS_2}, that if for some $T \subseteq U_A \times U_B$ and $D=(D_1,D_2) \in U_A \times U_B$ it holds that $(T,(D_1,D_2)) \in \mathcal{S}_1 \otimes \mathcal{S}_2$ where $D_A \preceq D_1$ and $D_B \preceq D_2$, then $T$ must include a pair $(M,N)$ such that $D_A \preceq M$ and $D_B \preceq N$. Since there exists a derivation of $(D_A,D_B)$ from $S$, there must exist some $j \in [t]$ for which $D_A \preceq M_j$ and $D_B \preceq N_j$, hence $D_A \otimes D_B \preceq M_j \otimes N_j$, as required.

We next show that $R_{\mathbb{B}}(A \otimes B) \geq R(\mathcal{S}_A \otimes \mathcal{S}_B)$.
Let $t = R_{\mathbb{B}}(A \otimes B)$. Then there exists a collection of $t$ matrices $M^{(1)},\ldots,M^{(t)}$ of Boolean rank $1$ that covers $A \otimes B$.
By Claim~\ref{claim:rectangles}, for every $i \in [t]$ there are two matrices $M^{(i)}_A$ and $M^{(i)}_B$ of Boolean rank $1$ such that $M^{(i)}_A \preceq A$, $M^{(i)}_B \preceq B$, and $M^{(i)} \preceq M^{(i)}_A \otimes M^{(i)}_B$. In particular, the $t$ matrices $M^{(i)}_A \otimes M^{(i)}_B$ with $i \in [t]$ cover $A \otimes B$.
We claim that the collection $S$ of pairs $(M^{(i)}_A,M^{(i)}_B)$ with $i \in [t]$ spans the set $U_A \times U_B$ in the spanoid $\mathcal{S}_A \otimes \mathcal{S}_B$.
This will imply that $R(\mathcal{S}_A \otimes \mathcal{S}_B) \leq  t$, completing the proof.

To this end, it suffices to show that for every two matrices $D_A$ and $D_B$ with precisely one nonzero entry such that $D_A \preceq A$ and $D_B \preceq B$,
the pair $(D_A,D_B)$ can be derived from $S$ in $\mathcal{S}_A \otimes \mathcal{S}_B$, because these pairs can be used to derive every other pair in $U_A \times U_B$.
Let $I$ be the collection of indices $i \in [t]$ satisfying $D_B \preceq M^{(i)}_B$.
Observe that for all $i \in I$,
\begin{eqnarray}\label{eq:derivation}
((M^{(i)}_A,M^{(i)}_B),(M^{(i)}_A,D_B)) \in \mathcal{S}_A \otimes \mathcal{S}_B,
\end{eqnarray}
and that the matrices $M^{(i)}_A$ with $i \in I$ cover $A$.
The latter holds because otherwise the copy of $A$ in $A \otimes B$ that corresponds to the nonzero entry of $D_B$ is not covered by the matrices $M^{(i)}_A \otimes M^{(i)}_B$ with $i \in [t]$.
So there exists an $i \in I$ satisfying $D_A \preceq M^{(i)}_A$, for which it holds that $((M^{(i)}_A,D_B),(D_A,D_B)) \in \mathcal{S}_A \otimes \mathcal{S}_B$.
Hence, by~\eqref{eq:derivation}, $S \models (D_A,D_B)$ in $\mathcal{S}_A \otimes \mathcal{S}_B$, as required.
\end{proof}

\subsection{An Upper Bound on the Rank of Products of Spanoids}
The notion of spanoids is much more general than that of spanoids associated with $0,1$-matrices.
And yet, our technique for proving upper bounds on the Boolean rank of Kronecker products of matrices, as given in Theorem~\ref{theorem:RankTensorIntro}, naturally extends to the wider framework of spanoids. For potential future applications, we present below the analogue of Theorem~\ref{theorem:RankTensorIntro} for general spanoids with a quick proof.

\BT
Let $\mathcal{S}_1$ and $\mathcal{S}_2$ be two spanoids on the sets $U_1$ and $U_2$ respectively.
Let $M_1,\ldots,M_s \subseteq U_1$ and $N_1,\ldots,N_s \subseteq U_2$ be sets satisfying that for every $i \in U_1$ it holds that $\span(\bigcup_{j:i \in M_j}{N_j}) = U_2$.
Then,
\[R(\mathcal{S}_1 \otimes \mathcal{S}_2) \leq \sum_{t=1}^{s}{|M_t| \cdot |N_t|}.\]
\ET

\BPF
It suffices to show that the set $S = \bigcup_{t=1}^{s}{(M_t \times N_t)}$ spans $U_1 \times U_2$ in the product spanoid $\mathcal{S}_1 \otimes \mathcal{S}_2$, because this yields that
\[ R(\mathcal{S}_1 \otimes \mathcal{S}_2) \leq |S| \leq \sum_{t=1}^{s}{|M_t| \cdot |N_t|}.\]
To see this, consider an arbitrary pair $(i,j) \in U_1 \times U_2$. By assumption, the set $T=\bigcup_{j':i \in M_{j'}}{N_{j'}}$ spans $U_2$.
In particular, $T \models j$ in $\mathcal{S}_2$.
By the definition of $\mathcal{S}_1 \otimes \mathcal{S}_2$, it follows that $\{i\} \times T \models (i,j)$ in $\mathcal{S}_1 \otimes \mathcal{S}_2$.
By our definition of $S$, we have $\{i\} \times T \subseteq S$, and thus $(i,j) \in \span(S)$.
We have shown that $S$ spans $U_1 \times U_2$ in $\mathcal{S}_1 \otimes \mathcal{S}_2$, so we are done.
\EPF

\section*{Acknowledgements}
We are grateful to LeRoy B.~Beasley and Pauli Miettinen for their help with the literature and to Eyal Kushilevitz for useful discussions~\cite{KushilevitzPersonal}.
We are also grateful to the anonymous referees for valuable suggestions.

\bibliographystyle{plain}
\bibliography{rankbib}

\end{document}